\documentclass[11pt,twoside]{article}
\usepackage[utf8]{inputenc}
\usepackage[english]{babel}
\usepackage{amsmath,amssymb,amsthm}
\usepackage{hyperref}

\usepackage[margin=1in]{geometry}
\newtheorem{theorem}{Theorem}[section]
\newtheorem{proposition}[theorem]{Proposition}
\newtheorem{lemma}[theorem]{Lemma}
\newtheorem{corollary}[theorem]{Corollary}

\theoremstyle{definition}
\newtheorem{definition}[theorem]{Definition}
\newtheorem{remark}[theorem]{Remark}

\numberwithin{equation}{section}

\begin{document}
	
	\title{\textbf{Non-ordinary primes and congruences for weakly holomorphic modular forms of level 2}}
	\author{\textbf{Nechaev Vladimir}}
	\maketitle
	
	\begin{abstract}
		\noindent In this paper, we investigate the arithmetic properties of the Fourier coefficients of weakly holomorphic modular forms of weight $k$ for the congruence subgroup $\Gamma_0(2)$, which have poles exclusively at the cusp $\infty$. Using the structure of canonical bases and the explicit duality between the coefficients of spaces of weakly holomorphic modular forms, which have a pole only at $\infty$, we prove a general family of congruences modulo primes $p \ge 5$ and derive several corollaries as examples of the theorem's application. \\
		\nolinebreak
	\end{abstract}
	
	\section{Introduction}
	
	The arithmetic of the coefficients of modular forms occupies an important place in modern number theory. This is particularly evident in the study of Ramanujan's tau-function $\tau(n)$, which is the coefficient function of the unique normalized cusp form of weight 12 with respect to $SL(2,\mathbb{Z})$:
	\[
	\Delta(z)=q\prod_{m\ge 1}(1-q^m)^{24} = \sum_{n\ge 1}\tau(n)q^n, \qquad q=e^{2\pi i z}.
	\]
	In his work, Lehmer \cite{Lehmer1947} formulated the famous conjecture on the non-vanishing of $\tau(n)$, he conjectured that $\tau(n)\neq 0$ for all $n\ge 1$. This problem remains open to this day, despite significant progress in related questions concerning the distribution and divisibility of modular form coefficients.
	
	Not much is known about the distribution of non-ordinary primes. Suppose that $f(z)=\sum_{n=1}^{\infty} a_f(n)q^n \in S_k$ is a normalized Hecke eigenform for $\mathrm{SL}_2(\mathbb{Z})$. Following Serre \cite{Serre1973}, we say that $f$ is non-ordinary at a prime $p$ if $Na_f(p) \equiv 0 \pmod p$, where, for an algebraic number $\alpha$, we have defined $N\alpha = N_{\mathbb{Q}(\alpha)/\mathbb{Q}}(\alpha)$. Such primes are of great importance in the theory of $p$-adic $L$-functions and congruences of eigenforms.
	
	There are relevant works concerning non-ordinary primes and weight divisibility. Jin, Ma, and Ono \cite{JinMaOno2016} proved that for any finite set of primes $S$, there exist normalized Hecke eigenforms of level $1$ that are non-ordinary at all primes in $S$. Furthermore, Ma \cite{Ma2025} established more precise results regarding $\Delta$: if $(p-1)\mid (k-12)$ for a prime $p \ge 5$, then at least one of the following is true: $f$ is non-ordinary at $p$, or $f \equiv \Delta \pmod p$. 
	
	Weakly holomorphic modular forms are actively used to prove such statements. Duke and Jenkins \cite{DukeJenkins2008} discovered an explicit duality between the coefficients of canonical bases for weights $k$ and $2-k$. This idea was later developed for genus zero levels: Garthwaite and Jenkins thoroughly investigated levels 2 and 3, constructing analogous bases and proving a series of congruences for their coefficients \cite{JenkinsGarthwaite2013}. In the works of Jenkins and Haddock \cite{JenkinsHaddock2014}, these methods were used for level 4, and subsequent research examined levels 6, 8, 10, and 12 \cite{Moss2014}. 
	
	In this paper, we investigate weakly holomorphic modular forms of level $2$ and their arithmetic properties modulo primes. We denote by $M_k^\#(2)$ the space of weakly holomorphic modular forms of weight $k$ for $\Gamma_0(2)$, which are holomorphic on the upper half-plane $\mathbb{H}$ and can have poles only at the cusp $\infty$.
	
	The main result of this paper is the following theorem:
	
	\begin{theorem}
		\label{theo}
		Let \(p \ge 5\) be a prime, and suppose that
		\[
		f(z)=\sum_{n\gg-\infty}^{\infty} a_f(n)q^n \in M_k^{\#}(2)\cap \mathcal{O}_L[[q]],
		\]
		where \(k\in 2\mathbb{Z}\) and \(\mathcal{O}_L\) is the ring of algebraic integers of a number field \(L\). Suppose $t \ge 2$ is an even integer and \(a\ge 0\) is an integer for which \(k-2\le tp^a\). If \(\operatorname{ord}_\infty(f)>-p^a\) and \((p-1)\mid (k-t-2)\), then for every integer \(b\ge a\) and every \(m\ge -l + 1 \), where $t = -4l - k'$ and $k' \in \{0, 2\} $, we have
		\[
		a_f(mp^b)-\sum_{r = l}^{0} a_{f^{(2)}_{t+2,r}}(m)a_{f}(-rp^b) \equiv 0 \pmod p,
		\]
		where $a_{f^{(2)}_{t+2,r}}(m)$ are the Fourier coefficients of the canonical basis elements for $M_{t+2}^{\#}(2)$.
	\end{theorem}
	
	Theorem \ref{theo}, allows us to deduce more about non-ordinary primes. Let $S_4(z)=\frac{E_4(z)-E_4(2z)}{240}$ be the weight 4 form for $\Gamma_0(2)$. We establish the following corollaries.
	
	\begin{corollary}
		\label{cor1}
		Under the assumptions of Theorem \ref{theo}, if $t=2$ ($l=-1$, $k'=2$), then for every integer \(b\ge a\) and every integer $m \ge 1$, we have
		\[
		a_f(mp^b)-\bigl( \sigma_{3}(m)-\sigma_{3}\left(\frac{m}{2}\right)\bigr)a_f(p^b) -240\sigma_{3}\left(\frac{m}{2}\right)a_f(0) \equiv 0 \pmod p,
		\]
		where $\sigma_{3}\bigl(\frac{m}{2}\bigr) = 0$ if $m$ is odd.
	\end{corollary}
	
	\begin{corollary}
		\label{cor2}
		Let \(p \ge 5\) be a prime, and let $f \in S_k(2)$ be a normalized Hecke eigenform. If $(p-1)\mid (k-4)$, at least one of the following is true.
		\begin{enumerate}
			\item $f$ is non-ordinary at $p$.
			\item $f \equiv S_4 \pmod p$.
		\end{enumerate}
	\end{corollary}
	
	\begin{corollary}
		\label{cor3}
		Let \(p \ge 5\) be a prime, and let $f \in S_k(2)$ be a normalized Hecke eigenform. If $(p-1)\mid (k-4)$ and $f$ is ordinary at $p$, then
		\[
		a_f(p) \equiv 1 \pmod p.
		\]
	\end{corollary}
	
	The paper is organized as follows. In Section \ref{sec2}, we provide the necessary background on weakly holomorphic forms and Eisenstein series. In Section \ref{sec3}, we outline the construction of the canonical basis for $M_k^\#(2)$ and prove the duality property. Section \ref{sec4} is devoted to the proof of Theorem \ref{theo} and its corollaries.
	
	\section{Preliminaries}
	\label{sec2}
	
	Let $\mathbb{H}$ be the upper half-plane. The set of cusps for $\Gamma_0(2)$, denoted by $\mathbb{P}^1(\mathbb{Q})/\Gamma_0(2)$, consists of two points, which can be represented by $0$ and $\infty$.
	
	We recall the definition of weakly holomorphic modular forms:
	\begin{definition}[\cite{Ono2004}]
		A function $f: \mathbb{H} \rightarrow \mathbb{C}$ is called a \textbf{weakly holomorphic modular form} of weight $k$ for $\Gamma_0(N)$ if it satisfies the weight $k$ slash operator invariance for all matrices in $\Gamma_0(N)$, is holomorphic on $\mathbb{H}$, and has poles (if any) supported only at the cusps. The space of such forms is denoted by $M_k^!(N)$. The forms that can have a pole exclusively at the cusp $\infty$ are called \textbf{totally weakly holomorphic forms} and their space is denoted by $M_k^\#(N)$.
	\end{definition}
	
	We will need the following property for weight 2:
	
	\begin{proposition}
		\label{prop:residue}
		If $h \in M_2^\#(2)$ and $h(0) = 0$, then its constant term $a_h(0)$ vanishes.
	\end{proposition}
	
		\begin{proof}
		Let $\omega = h(z)\,dz$. Since $h$ has weight 2, $\omega$ is invariant under $\Gamma_0(2)$ and defines a meromorphic differential on the compact Riemann surface $X_0(2) = \Gamma_0(2)\backslash(\mathbb{H} \cup \{0, \infty\})$. By the residue theorem on compact Riemann surfaces, we have $\sum \operatorname{Res}(\omega) = 0$. 
		
		Because $h \in M_2^\#(2)$ is holomorphic on $\mathbb{H}$ and vanishes at the cusp 0, the only possible pole is at $\infty$. By taking the local coordinate $q = e^{2\pi i z}$, we obtain:
		\[
		\omega = h(z)\,dz = \frac{1}{2\pi i} \sum_{n \gg -\infty}^{\infty} a_h(n) q^{n-1}\,dq.
		\]
		The residue at infinity is $\operatorname{Res}_\infty(\omega) = \frac{1}{2\pi i} a_h(0)$. Since the sum of all residues is zero, we get $a_h(0)=0$.
	\end{proof}
	
	We also recall a well-known congruence for Eisenstein series.
	
	\begin{proposition}
		\label{prop:eisenstein}
		If $p\ge 5$ is a prime, then we have the congruence:
		\[
		E_{p-1}(z)\equiv 1 \pmod p.
		\]
	\end{proposition}
	
	\begin{proof}
		For prime $p\ge 5$, the normalized Eisenstein series is given by
		\[
		E_{p-1}(z)=1-\frac{2(p-1)}{B_{p-1}}\sum_{n=1}^{\infty}\sigma_{p-2}(n)q^n.
		\]
		Applying the von Staudt-Clausen theorem, the denominator of the Bernoulli number $B_{p-1}$ is divisible by $p$. Thus $\frac{1}{B_{p-1}}\equiv 0 \pmod p$. The conclusion follows.
	\end{proof}
	
	Next, we refer to a property for Hecke eigenforms, which will be used in the proof.
	
	\begin{proposition}
		\label{prop:hecke_mult}
		If $f \in S_k(2)$ is a normalized Hecke eigenform, then for any prime $p$ and $m \ge 1$:
		\[
		a_f(p)a_f(p^m) = a_f(p^{m+1}) + p^{k-1}a_f(p^{m-1}) \equiv a_f(p^{m+1}) \pmod p.
		\]
		By induction, it follows that $a_f(p^m) \equiv a_f(p)^m \pmod p$.
	\end{proposition}
	
	\section{Canonical basis of the space $M_k^\#(2)$ and coefficient duality}
	\label{sec3}
	
	We present the construction of the canonical basis for $M_k^\#(2)$. It relies on three fundamental functions \cite{Guindy2009}:
	\[
	\psi(z)=\left(\frac{\eta(z)}{\eta(2z)}\right)^{24} = q^{-1}-24+276q+\cdots \in M_0^\#(2),
	\]
	\[
	F_2(z)=2E_2(2z)-E_2(z)\in M_2(2), \qquad S_4(z)=\frac{E_4(z)-E_4(2z)}{240}\in M_4(2).
	\]
	Here, $\psi(z)$ is a Hauptmodul for $\Gamma_0(2)$, the form $S_4(z)$ vanishes at $\infty$, and $F_2(z)$ is a normalized holomorphic form of weight $2$.
	
	Let $k \in 2\mathbb{Z}$ be written as $k=4\ell+k'$, where $k'\in\{0,2\}$ and $\ell\in\mathbb{Z}$. The space $M_k^\#(2)$ admits a unique canonical basis $\{f^{(2)}_{k,n}(z)\}_{n\ge -\ell}$, where each element satisfies
	\[
	f^{(2)}_{k,n}(z)=q^{-n} + \sum_{m \ge \ell+1} a_k(m,n)q^m.
	\]
	Each element of this basis can be expressed as $f^{(2)}_{k,n}(z)=S_4^\ell(z)\,F_{k'}(z)\,P_{n}(\psi(z))$, where $P_n(x)$ is a polynomial with integer coefficients of degree $n+\ell$.
	
	Furthermore, we define the dual family $g^{(2)}_{k,n}(z)\in M_k^\#(2)$, which is characterized by vanishing at the cusp $0$. 
	
    Such a family is also constructed recursively: starting with
	\[
	g^{(2)}_{k,-\ell+1}(z)=S_4^\ell(z)\,\psi(z)\,F_{k'}(z),
	\]
	and then multiplying again by $\psi(z)$ and subtracting the previous terms,
	to set the intermediate Fourier coefficients to zero.
	
	The Fourier expansion is given by
	\[
	g^{(2)}_{k,n}(z)=q^{-n}+\sum_{m\ge \ell} b_k(m,n)\,q^m, \qquad n\ge -\ell+1.
	\]
	
	The key property of these spaces, which allows us to build explicit congruences, is the symmetry of their coefficients.
	
	\begin{lemma}
		For all integers $m$ and $n$, we have
		\[
		a_k(m,n) = -b_{2-k}(n,m).
		\]
	\end{lemma}

	\begin{proof}
		Applying the standard identities from \cite [Theorem 1.2]{Guindy2009}, the generating function for the basis $f^{(2)}_{k,n}$ is
		\[
		\sum_{n=-\ell}^{\infty} f^{(2)}_{k,n}(\tau)\,q^n = \frac{(S_4^\ell F_{k'})(z)}{(S_4^\ell F_{k'})(\tau)} \cdot \frac{\psi(\tau)F_2(\tau)}{\psi(\tau)-\psi(z)}.
		\]
		For the dual family $g^{(2)}_{k,n}$, we have
		\[
		\sum_{n=-\ell+1}^{\infty} g^{(2)}_{k,n}(\tau)\,q^n = \frac{(S_4^\ell \psi F_{k'})(z)}{(S_4^\ell \psi F_{k'})(\tau)} \cdot \frac{\psi(\tau)F_2(\tau)}{\psi(\tau)-\psi(z)}.
		\]
		Swapping the variables $z$ and $\tau$, and replacing the weight $k$ with $2-k$, the expression transforms into $-\sum_{m\ge \ell} f^{(2)}_{2-k,m}(\tau)\,q^m$. Thus, by comparing coefficients, we obtain $a_k(m,n) = -b_{2-k}(n,m)$. The lemma is proven.
	\end{proof}
	
	\section{Proof of the main results}
	\label{sec4}
	
	\begin{proof}[Proof of Theorem \ref{theo}]
		Note that the divisibility condition $(p-1)\mid (k-t-2)$ implies the congruence:
		\[
		(k-2)-tp^b \equiv k-(t+2) \equiv 0 \pmod{p-1}.
		\]
		Since $k-2 \le tp^a \le tp^b$ for any $b \ge a$ there exists a non-negative integer $c$ satisfying:
		\[
		-tp^b + c(p-1) = 2-k.
		\]
		Applying the construction from Sect. 3, consider $g_{-t,m}(z)$ from the dual family in $M_{-t}^\#(2)$. Its $q$-expansion is given by
		\[
		g_{-t,m}(z) = q^{-m} - \sum_{r=\ell}^{0} a_{f^{(2)}_{t+2,r}}(m)q^r + \sum_{n=1}^{\infty} a_{g_{-t,m}}(n)q^n.
		\]
		Thus we establish that
		\[
		\mathcal{G}(z) = g_{-t, m}^{p^b}(z)E_{p-1}^c(z) \in M_{2-k}^\#(2).
		\]
		Since $g_{-t, m}(0) = 0$, we have $\mathcal{G}(0) = 0$. For any $f \in M_k^\#(2)$, the product $h(z) = \mathcal{G}(z)f(z)$ lies in $M_2^\#(2)$ and vanishes at the cusp $0$. Applying Proposition 2.1, we obtain that the constant term $a_h(0)$ of $h(z)$ is zero.
		
		Knowing that $E_{p-1}(z) \equiv 1 \pmod p$ from Proposition 2.2, and with Fermat's little theorem, we get
		\[
		h(z) \equiv \left( q^{-mp^b} - \sum_{r=\ell}^{0} a_{f^{(2)}_{t+2,r}}(m)q^{rp^b} + O(q^{p^b}) \right) \left( \sum_{n \gg -\infty}^{\infty} a_f(n)q^n \right) \pmod p.
		\]
		As $\operatorname{ord}_\infty(f) > -p^a \ge -p^b$, the terms in $O(q^{p^b})$ do not contribute to the constant term of the product. Therefore, the constant term must satisfy the congruence
		\[
		0 = a_h(0) \equiv a_f(mp^b) - \sum_{r=\ell}^{0} a_{f^{(2)}_{t+2,r}}(m)a_f(-rp^b) \pmod p.
		\]
		This completes the proof.
		\end{proof}
	
	\begin{proof}[Proof of Corollary \ref{cor1}]
		Take $t=2$, $\ell=-1$, and $k'=2$ in Theorem \ref{theo}. We get
		\[
		a_f(mp^b) - a_{f^{(2)}_{4,-1}}(m)a_f(p^b) - a_{f^{(2)}_{4,0}}(m)a_f(0) \equiv 0 \pmod p.
		\]
		Note that $f^{(2)}_{4,-1}(z) = S_4(z) = \sum_{n=1}^\infty ( \sigma_3(n) - \sigma_3(n/2) )q^n$, and $f^{(2)}_{4,0}(z) = E_4(2z) = 1 + 240\sum_{n=1}^\infty \sigma_3(n)q^{2n}$. Substituting these explicit coefficients, we get the conclusion.
	\end{proof}
	
	\begin{proof}[Proof of Corollary \ref{cor2}]
		For any normalized Hecke eigenform $f \in S_k(2)$, we have $a_f(0)=0$. Corollary 1.2 implies
		\begin{equation}
			\label{eq:eigen_congruence}
			a_f(mp) - \left(\sigma_3(m) - \sigma_3\left(\frac{m}{2}\right)\right)a_f(p) \equiv 0 \pmod p.
		\end{equation}
		Assuming $f$ is ordinary at $p$, i.e., $a_f(p) \not\equiv 0 \pmod{\mathfrak{p}}$ for any prime ideal $\mathfrak{p}$ above $p$, we show $f \equiv S_4 \pmod p$ in the following. 
		
		Take $m=p$ first. Since $p \ge 5$ is odd, $\sigma_3(p/2)=0$. Then \eqref{eq:eigen_congruence} turns to $a_f(p^2) - \sigma_3(p)a_f(p) \equiv 0 \pmod p$. Applying Proposition 2.3, this is equal to
		\[
		a_f(p)\left(a_f(p) - \sigma_3(p)\right) \equiv 0 \pmod p.
		\]
		Under the assumption that $f$ is ordinary at $p$, we obtain $a_f(p) \equiv \sigma_3(p) \pmod p$. By induction, $a_f(p^c) \equiv \sigma_3(p^c) \pmod p$ for every $c \in \mathbb{Z}_{+}$.
		
		For those $m$ coprime to $p$, knowing that $a_f(n)$ is a multiplicative function, \eqref{eq:eigen_congruence} tells us $a_f(p)(a_f(m) - \sigma_3(m)) \equiv 0 \pmod p$. Hence $a_f(m) \equiv \sigma_3(m) \pmod p$. With the multiplicativity of $a_f(n)$ and $\sigma_3(n)$, we obtain that $a_f(n) \equiv \sigma_3(n) - \sigma_3(n/2) \pmod p$ holds for every $n \in \mathbb{Z}_{+}$. That is, $f \equiv S_4 \pmod p$.
	\end{proof}
	
	\begin{proof}[Proof of Corollary \ref{cor3}]
		Since $f$ is ordinary at $p$, Corollary 1.3 implies $f \equiv S_4 \pmod p$. The coefficient of $q^p$ in $S_4(z)$ is $\sigma_3(p) = 1 + p^3 \equiv 1 \pmod p$. Thus we get $a_f(p) \equiv 1 \pmod p$.
	\end{proof}
	\begin{remark}
		The proof of Theorem 1 can be extended to other congruence subgroups $\Gamma_0(N)$ of genus 0. The main difficulty lies in finding two families with duality in the coefficients - the rest follows by analogy.
	\end{remark}


\begin{thebibliography}{99}
		
		\bibitem{Lehmer1947} 
		D.~H. Lehmer, \textit{The vanishing of Ramanujan's function $\tau(n)$}, Duke Math. J. \textbf{14}(2), 429--433 (1947).
		
		\bibitem{Serre1973} 
		J.-P. Serre, \textit{A Course in Arithmetic}, Graduate Texts in Mathematics, Vol. 7. Springer, New York (1973).
		
		\bibitem{JinMaOno2016} 
		S. Jin, W. Ma, K. Ono, \textit{A note on non-ordinary primes}, Proc. Amer. Math. Soc. \textbf{144}(11), 4591--4597 (2016).
		
		\bibitem{Ma2025} 
		W. Ma, \textit{Non-ordinary primes and $\Delta$ congruences for modular forms}, Ramanujan J. \textbf{68}(2), 39--47 (2025).
		
		\bibitem{DukeJenkins2008} 
		W. Duke, P. Jenkins, \textit{On the coefficients and zeros of weakly holomorphic modular forms}, Pure Appl. Math. Q. \textbf{4}(4), 1327--1340 (2008).
		
		\bibitem{JenkinsGarthwaite2013} 
		S.~A. Garthwaite, P. Jenkins, \textit{Zeros of weakly holomorphic modular forms of levels 2 and 3}, Math. Res. Lett. \textbf{20}(4), 657--674 (2013).
		
		\bibitem{JenkinsHaddock2014} 
		A. Haddock, P. Jenkins, \textit{Zeros of weakly holomorphic modular forms of level 4}, Int. J. Number Theory \textbf{10}(2), 455--470 (2014).
		
		\bibitem{Moss2014}
		V. Iba, P. Jenkins, M. Warnick, \textit{Divisibility properties of coefficients of modular functions in genus zero levels}, Integers \textbf{19}, A7 (2019).
		
		\bibitem{Ono2004} 
		K. Ono, \textit{The Web of Modularity: Arithmetic of the Coefficients of Modular Forms and $q$-series}, CBMS Regional Conference Series in Mathematics, Vol. 102. Amer. Math. Soc., Providence (2004).
		
		\bibitem{Guindy2009} 
		A. El-Guindy, \textit{Fourier expansions with modular form coefficients}, Int. J. Number Theory \textbf{5}(8), 1433--1446 (2009).
		
	\end{thebibliography}
\end{document}